\documentclass[10pt,a4paper,reqno]{amsart}
\usepackage{epsfig}
\usepackage{color}
\usepackage{amssymb,amsmath,amsthm,amstext,amsfonts}
\usepackage{amssymb,amsmath,amsthm,amstext,amsfonts}
\usepackage{amsmath,amstext,amsthm,amsfonts}
\usepackage{color}
\usepackage[mathscr]{eucal}

\usepackage{enumitem}

\usepackage{hyperref}
\RequirePackage[dvipsnames]{xcolor} % [dvipsnames]
\definecolor{halfgray}{gray}{0.55} 
\definecolor{webgreen}{rgb}{0,0.5,0}
\definecolor{webbrown}{rgb}{.6,0,0} \hypersetup{%
	colorlinks=true, linktocpage=true, pdfstartpage=3,
	pdfstartview=FitV,%
	breaklinks=true, pdfpagemode=UseNone, pageanchor=true,
	pdfpagemode=UseOutlines,%
	plainpages=false, bookmarksnumbered, bookmarksopen=true,
	bookmarksopenlevel=1,%
	hypertexnames=true,
	pdfhighlight=/O,%hyperfootnotes=true,%nesting=true,%frenchlinks,%
	urlcolor=webbrown, linkcolor=RoyalBlue,
	citecolor=webgreen, %pagecolor=RoyalBlue,%
	pdftitle={},%
	pdfauthor={},%
	pdfsubject={2000 MAthematical Subject Classification: Primary:},%
	pdfkeywords={},%
	pdfcreator={pdfLaTeX},%
	pdfproducer={LaTeX with hyperref}%
}

\usepackage{psfrag}
\usepackage{url}
\usepackage{epstopdf}
\usepackage{epic,eepic}
\usepackage{upgreek}
\usepackage{amssymb}
\usepackage{mathrsfs}
\usepackage{verbatim}
\usepackage{mathrsfs}
\usepackage{amstext}
\usepackage{amsthm}
\usepackage{amssymb}
\usepackage{graphicx}

\usepackage[]{mdframed}

\usepackage{hyperref}
\RequirePackage[dvipsnames]{xcolor} % [dvipsnames]
\definecolor{halfgray}{gray}{0.55} 
\definecolor{webgreen}{rgb}{0,0.5,0}
\definecolor{webbrown}{rgb}{.6,0,0} \hypersetup{%
	colorlinks=true, linktocpage=true, pdfstartpage=3,
	pdfstartview=FitV,%
	breaklinks=true, pdfpagemode=UseNone, pageanchor=true,
	pdfpagemode=UseOutlines,%
	plainpages=false, bookmarksnumbered, bookmarksopen=true,
	bookmarksopenlevel=1,%
	hypertexnames=true,
	pdfhighlight=/O,%hyperfootnotes=true,%nesting=true,%frenchlinks,%
	urlcolor=webbrown, linkcolor=RoyalBlue,
	citecolor=webgreen, %pagecolor=RoyalBlue,%
	pdftitle={},%
	pdfauthor={},%
	pdfsubject={2000 MAthematical Subject Classification: Primary:},%
	pdfkeywords={},%
	pdfcreator={pdfLaTeX},%
	pdfproducer={LaTeX with hyperref}%
}

  [1995/08/01 v0.1 Double stroke roman fonts]

\def\ds@whichfont{dsrom}
\relax

\DeclareMathAlphabet{\mathds}{U}{\ds@whichfont}{m}{n}

\newtheorem{theorem}{Theorem}[section]
\newtheorem{lemma}[theorem]{Lemma}
\newtheorem{corollary}[theorem]{Corollary}

\newtheorem{proposition}[theorem]{Proposition}
\theoremstyle{definition}
\newtheorem{definition}[theorem]{Definition}

\newtheorem{remark}[theorem]{Remark}

\theoremstyle{plain}

\theoremstyle{plain}
\theoremstyle{plain}
\theoremstyle{remark}
\newtheorem*{acknowledgement*}{Acknowledgement}

\DeclareMathOperator{\Fix}{Fix}
\DeclareMathOperator{\Per}{Per}
\DeclareMathOperator{\ra}{r}
\DeclareMathOperator{\res}{r_{es}}

\newcommand{\cA}{{\mathcal A}}
\newcommand{\cB}{{\mathcal B}}
\newcommand{\cG}{{\mathcal G}}
\newcommand{\cM}{{\mathcal M}}
\newcommand{\cN}{{\mathcal N}}
\newcommand{\cR}{{\mathcal R}}

\newcommand{\cE}{{\mathcal E}}
\newcommand{\cF}{{\mathcal F}}
\newcommand{\cH}{{\mathcal H}}

\newcommand{\Id}{\text{Id}}
\newcommand{\N}{\mathbb{N}}
\newcommand{\R}{\mathbb{R}}

\def\Reg{\mathcal{R}^{\mu}}

\usepackage{orcidlink}

\date{\today}

\begin{document}
\title[Dominated splittings and periodic data]{Dominated splittings and periodic data for quasi-compact operator cocycles}

\author[Lucas Backes]{Lucas Backes \orcidlink{0000-0003-3275-1311}}
\address{\noindent Departamento de Matem\'atica, Universidade Federal do Rio Grande do Sul, Av. Bento Gon\c{c}alves 9500, CEP 91509-900, Porto Alegre, RS, Brazil.}
\email[]{lucas.backes@ufrgs.br}


\maketitle

\begin{abstract}
For infinite-dimensional quasi-compact cocycles over a map satisfying a certain closing condition, we show that periodic orbits carry enough information to guarantee the existence of a dominated splitting. More precisely, we establish that if the moduli of the $(k+1)$-largest eigenvalues of the cocycle are $e^{\lambda_1n}\geq e^{\lambda_2n}\geq \ldots\geq e^{\lambda_kn}\geq e^{\lambda_{k+1}n}$ at every periodic point of period $n$, and $\lambda_k>\lambda_{k+1}$, then the cocycle admits a dominated splitting of index $k$. As a consequence, if $\lambda_k>0>\lambda_{k+1}$ then the cocycle is uniformly hyperbolic. Furthermore, we are able to obtain these same conclusions even when the eigenvalues are only close to constant, not strictly constant.
\end{abstract}

\section{Introduction}

\emph{(Uniform) Hyperbolicity} is a fundamental concept in the modern theory of dynamical systems. For instance, it serves as a model for robust chaotic dynamics and characterizes structurally stable systems.
Simply put, a system is said to be hyperbolic if the space can be split into two complementary directions such that along one direction, the system exhibits exponential expansion over time, while along the other, it shows exponential contraction. 
Despite its importance, (uniform) hyperbolicity was soon understood to be a less universal property than initially believed. It was shown that open sets of nonhyperbolic systems exist, and many real-world phenomena cannot be described within a hyperbolic framework. As a result, weaker forms of hyperbolicity were introduced to encompass a wider variety of systems. These include \emph{nonuniform hyperbolicity}, \emph{partial hyperbolicity}, and the more general concept of a \emph{dominated splitting}.
On the other hand, verifying whether a system exhibits any of these forms of hyperbolicity can be very complicated. Thus, a key problem in the field is finding different ways of characterizing these properties. 
In the present work, we are interested in presenting sufficient conditions under which a system exhibits a dominated splitting in the context of infinite-dimensional semi-invertible cocycles.

Given an invertible map $f\colon M\to M$ acting on compact metric space $(M,d)$ and a map $A\colon M\to \cB(X)$, where $\cB(X)$ denotes the space of all bounded linear operators of a Banach space $X$, we say that the \emph{cocycle} $A$ over $f$ admits a dominated splitting if, for each $x\in M$, there exists a continuous splitting $X=\cE(x)\oplus \cF(x)$ satisfying $A(x)\cE(x)=\cE(f(x))$ and $A(x)\cF(x)\subseteq \cF(f(x))$ and such that the
``greatest expansion'' of $A(x)$ along $\cF(x)$ is smaller than the ``greatest contraction'' of $A(x)$ along $\cE(x)$ by a factor that becomes exponentially small as time evolves (see Section \ref{sec: dominated splitting} for precise definition). The dominated splitting is said to have index $k$ if $\dim \cE(x) = k$ for some $x\in M$.
In Theorem \ref{thm: main quasi-compact}, we describe sufficient conditions under which $(A,f)$ admits a dominated splitting in the case when $f$ satisfies a certain ``closing property'' and the cocycle is quasi-compact.
Our conditions are expressed in terms of properties of the cocycle on periodic orbits. More precisely, we establish that if the moduli of the $(k+1)$-largest eigenvalues of the cocycle are $e^{\lambda_1n}\geq e^{\lambda_2n}\geq \ldots\geq e^{\lambda_kn}\geq e^{\lambda_{k+1}n}$ at \emph{every} periodic point of period $n$, and $\lambda_k>\lambda_{k+1}$, then the cocycle admits a dominated splitting of index $k$. As a consequence, if $\lambda_k>0>\lambda_{k+1}$ then the cocycle is uniformly hyperbolic. Furthermore, in Theorem \ref{thm: main narrow}, we observe that the same conclusions of Theorem \ref{thm: main quasi-compact} remain valid even when the eigenvalues are only close to constant, not strictly constant.

Our results were inspired by the work of DeWitt and Gogolev \cite{DG}, who obtained similar results in the context of \emph{invertible} and \emph{finite-dimensional} linear cocycles. Our results generalize those in \cite{DG} in two main directions: we consider the case of possibly \emph{non-invertible} cocycles acting on a Banach space of possible \emph{infinite dimension}. The proof consists in showing that our hypothesis imply that a certain condition,  which is equivalent to the existence of a dominated splitting, is satisfied. This condition is given in terms of the behavior of the Gelfand numbers of the cocycle along orbits and was obtained by Blumenthal and Morris \cite{BM} generalizing previous work of Bochi and Gourmelon \cite{BG} in the finite-dimensional context. Additionally, we rely heavily on the machinery developed by Blumenthal \cite{B16} and González-Tokman and Quas \cite{GTQ} to prove an infinite-dimensional version of Oseledets' Mutliplicative Ergodic Theorem.

Dominated splittings were introduced in the 1980s and played a crucial role in the study of the Stability Conjecture \cite{Liao, Mane84, Pliss}. Since then, the theory has been significantly developed, and in some contexts, its dynamics are quite well-understood. For example, in their important work, Pujals and Sambarino \cite{PS} provide a clear description of two-dimensional dynamics that admit a dominated splitting. For a more comprehensive overview of dominated splittings in the context of smooth dynamics, we recommend the excellent survey by Pujals \cite{Pujals}. Returning to the setting of linear cocycles, we note that the problem of finding sufficient conditions for the existence of a dominated splitting - the main aim of this paper - has already been addressed by other authors.
For instance, in addition to the work of DeWitt and Gogolev \cite{DG} mentioned earlier, significant contributions were made by Kassel and Potrie \cite{KP}. They have shown that, under suitable conditions on the base dynamics, a uniform gap between the $k$-th and $(k+1)$-th Lyapunov exponents for \emph{all invariant measures} implies the existence of a dominated splitting of index $k$.
A similar result was obtained by Yemini in \cite{Yemini}. He showed that if a uniform and $C^0$-persistent gap exists between the $k$-th and $(k+1)$-th Lyapunov exponents (associated with a fixed invariant measure), a dominated splitting of index $k$ exists on the support of that invariant measure. We emphasize that all of these works are set in a finite-dimensional and invertible context, whereas ours is placed in a potentially infinite-dimensional and non-invertible setting.

The paper is organized as follows: Section \ref{sec: statements} presents the setting and our main results. Section \ref{sec: prelim} provides preliminary results useful in Section \ref{sec: proof theo compact}, where the main theorems are proved. Finally, Section \ref{sec: applications} discusses applications and consequences of our results.

\section{Statements}\label{sec: statements}

Let $(M,d)$ be a compact metric space and $f\colon M \to M$ a homeomorphism. Denote by $\cM(f)$ the set of all $f$-invariant probability measures defined on the Borel subsets of $M$ and by $\cM_{erg}(f)$ the ergodic elements of $\cM(f)$. Moreover, let $(X,\|\cdot\|)$ be an arbitrary Banach space and denote by $\cB (X)$ the space of all bounded linear maps from $X$ to itself, by $\cB_0(X)$ the subset of $\cB(X)$ formed by the compact operators of $X$, by $\cB^i(X)$ the elements of $\cB(X)$ that are injective and $\cB_0^i(X)=\cB_0(X)\cap \cB^i(X)$. We recall that $\cB(X)$  is  a Banach space with respect to the operator norm
\[
\|T\| =\sup \lbrace \|Tv\|/\|v\|;\; \|v\|\neq 0 \rbrace, \quad T\in \cB (X)
\]
and $\cB_0(X)$ is a closed subspace of $(\cB(X),\|\cdot\|)$. Observe that although we use the same notation for the norms on $X$ and $\cB(X)$, this will not cause any confusion. 

We denote by $\cG(X)$ the Grassmanian of closed complemented subspaces of $X$ and by $\cG_k(X)$ and $\cG^k(X)$ the Grassmanians of $k$-dimensional and closed $k$-codimensional subspaces of $X$, respectively.

\subsection{Closing property}
We say that $f$ satisfies the \textit{closing property} if there exist $c,\theta ,l >0$ such that for any $x\in M$ and $n\in \N$, there exists $0\leq j\leq l$ and a periodic point $p\in M$ such that $f^{n+j}(p)=p$ and
\begin{equation}\label{eq: closing property}
d(f^i(x),f^i(p))\leq c e^{-\theta \min\lbrace i, n-i\rbrace}d(f^n(x),x),
\end{equation}
for every $i=0,1,\ldots ,n$. We note that shifts of finite type and locally maximal hyperbolic sets are particular examples of maps satisfying this property. We refer to~\cite[Ch. 18]{KH95} for details.

\subsection{Semi-invertible operator cocycles}
Given a map $A\colon M\to \cB(X)$, the \emph{semi-invertible operator cocycle} (or just \emph{cocycle} for short) generated by $A$ over $f$ is  defined as the map $\cA\colon \mathbb{N}\times M\to \cB(X)$ given by
\begin{equation*}\label{def:cocycles}
A^n(x):=\cA(n, x)=
\left\{
	\begin{array}{ll}
		A(f^{n-1}(x))\ldots A(f(x))A(x)  & \mbox{if } n>0 \\
		\Id & \mbox{if } n=0,\\
		%A(f^{n}(x))^{-1}\ldots A(f^{-1}(x))^{-1}& \mbox{if } n<0 \\
	\end{array}
\right.
\end{equation*}
for all $x\in M$. 
The term `semi-invertible' refers to the fact that the action of the underlying dynamical system $f$ is assumed to be invertible while the action on the fibers given by $A$ may fail to be invertible. 

We say that the cocycle generated by $A$ over $f$ is \emph{compact} if $A$ take values in $\cB_0(X)$, i.e. if $A(x)\in \cB_0(X)$ for each $x\in M$.

\subsection{Dominated splitting}\label{sec: dominated splitting}
 We say that $A$ admits a \emph{dominated splitting} if there exist continuous functions $\cE,\cF \colon X \to \cG(X)$ and constants $C>0$ and $\tau \in (0,1)$ such that $X=\cE(x)\oplus \cF(x)$ for every $x \in M$, $A^n(x)\cE(x)=\cE(f^n(x))$ and $A^n(x)\cF(x)\subseteq \cF(f^n(x))$ for every $x \in M$ and $n \geq 1$, and, moreover,
\[|A^n(x) v|\leq C\tau^n |A^n(x) u| \]
for every $u \in \cE(x)$, $v \in \cF(x)$ with $|u|=|v|=1$ and every $x \in M$ and $n \geq 1$. We say that the dominated splitting has \emph{index $k$} if $\dim \cE(x) = k$ for some (and, consequently, for all) $x\in M$.

\subsection{Quasi-compact operators and their spectrum} \label{sec: quasi-compact} Given $T\in \cB(X)$, let us denote by $\ra(T)$ and $\res(T)$ the \emph{spectral radius} and the \emph{essential spectral radius} of $T$, respectively. Recall that $T\in \cB(X)$ is called \emph{quasi-compact} if $\res(T)<\ra(T)$. Equivalently, $T\in \cB(X)$ is quasi-compact if $X$ can be decomposed as the sum of two $T$-invariant closed subspaces
\[X=Y \oplus Z\]  
such that $\ra(T_{|Z})<\ra(T)$, $\dim Y <+\infty$ and all the eigenvalues of $T_{|Y}$ have modulus $\ra(T)$ (see \cite{MMB}). In particular, given $\res(T)<\rho\leq \ra(T)$, the spectrum $\sigma(T)$ of $T$ contains only finitely many elements of modulus bigger than or equal to $\rho$ and, moreover, these elements are all eigenvalues of $T$. Let us denote them by $(\gamma_i)_{i=1}^{t_\rho}$. Then there exist $T$-invariant subspaces $Y_\rho$ and $Z_\rho$ such that
\[X=Y_\rho \oplus Z_\rho\]  
with $\dim Y_\rho<+\infty$, $Y_\rho=\bigoplus_{i=1}^{t_\rho}\left(\cup_{j\geq 1}\ker(T-\gamma_i)^j\right)$ and $Z_\rho=\{x\in X: \; \limsup_n \|T^nx\|^{1/n}<\rho\}$ (see \cite[Chapter XIV]{HH}). As a consequence, there are at most countably many elements in $\sigma(T)$ with modulus bigger than $\res(T)$ and these elements are all eigenvalues of $T$. Let us call them the \emph{exceptional eigenvalues of $T$}.

\subsection{Constant periodic data} \label{sec: constant pd qc} Given $A\colon M\to \cB(X)$, $n\in \N$ and $p\in \Fix(f^n)$, let us assume that $A^n(p)$ is a quasi-compact operator. 
Then, by the observations in Section \ref{sec: quasi-compact}, the exceptional eigenvalues of $A^n(p)$ repeated according to multiplicity can be written as $(\gamma_i(p))_{i=1}^{b_p}$ for some $b_p\in \N\cup \{+\infty\}$ with 
\[
|\gamma_1(p)|\geq|\gamma_2(p)|\geq |\gamma_3(p)|\geq \ldots . \]

\begin{definition}
Let $A\colon M\to \cB(X)$ be such that $A^n(p)$ is a quasi-compact operator for every $p\in \Fix(f^n)$ and $n\in \N$.
Given $k\leq \dim X $ and a sequence of real numbers $\lambda_1\geq \lambda_2\geq \cdots \geq \lambda_k$ (observe that $k$ may be infinite), we say that the cocycle $A$ has \emph{$k$-constant periodic data of type $(\lambda_1,\ldots,\lambda_k)$} if $b_p\geq k$ for every $p\in \Per(f)$ and
\[ |\gamma_i(p)|=e^{n\lambda_i} \text{ for all } 1\leq i\leq k, n\in \N \text{ and } p\in \Fix(f^n). \]
\end{definition}

\subsection{Quasi-compact cocycles}\label{sec: quasicompact cocycles}
 Let $B_X(0,1)$ denote the unit ball in $X$ centered at $0$. Given $T\in \cB(X)$, let $\|T\|_\text{ic}$ be the infimum over all $r>0$ with the property that $T(B_X(0,1))$ can be covered by finitely many open balls of radius $r$. It is easy to show that
\begin{equation}\label{ic1}
 \|T\|_\text{ic} \le \lVert T\rVert, \quad \text{for every $T\in \cB(X)$}
\end{equation}
and 
\begin{equation}\label{ic2}
 \|T_1 T_2\|_\text{ic} \le \|T_1\|_\text{ic} \cdot \|T_2\|_\text{ic}, \quad \text{for every $T_1, T_2\in \cB(X)$.}
\end{equation}
Hence, given $\mu\in \cM_{erg}(f)$, \eqref{ic2} together with Kingman's Subadditive Theorem (see, for instance, \cite{LLE}) implies that there exists $\kappa(\mu)\in [-\infty, \infty)$ such that
\[
 \kappa(\mu)=\lim_{n\to \infty}\frac 1n \log \| A^n(x)\|_{ic} \quad \text{for $\mu$-a.e. $x\in M$.}
\]
Observe that whenever $\dim X<\infty$ or, more generally, if $A$ takes values in the set of compact operators $\cB_0(X)$ of $X$, we have that $\kappa(\mu)=-\infty$. Indeed, in this case we have that $\| A^n(x)\|_{ic}=0$ for each $n$ which readily implies that
$\kappa(\mu)=-\infty$ (taking $\log 0=-\infty$).

In addition, by using again Kingman's Subadditive Theorem together with the subadditivity of the operator norm, we have that there exists $\lambda(\mu)\in [-\infty, \infty)$ such that
\[
 \lambda(\mu)=\lim_{n\to \infty}\frac 1n \log \lVert A^n(x)\rVert \quad \text{for $\mu$-a.e. $x\in M$.}
\]
Note  that~\eqref{ic1} implies that $\kappa(\mu) \le \lambda(\mu)$. We say that the cocycle $A$ is \emph{quasi-compact} with respect to $\mu$ if $\kappa(\mu)<\lambda(\mu)$. Quasi-compactness plays an important role in the study of ergodic theory associated to cocycles acting in infinite-dimensional Banach spaces. One of the main reasons for this is that, under a quasi-compactness assumption, we have the availability of the Multiplicative Ergodic Theorem (see, for instance, \cite{B16,FLQ13,GTQ0,GTQ}).
 Sufficient conditions under which the cocycle is quasi-compact are given, for instance, in~\cite[Lemma 2.1]{DFGTV}.

Observe that, given $p\in \Fix(f^n)$ and considering $\mu_p\in \cM_{erg}(f)$ given by
\[\mu_p=\frac{1}{n}\sum_{j=0}^{n-1}\delta_{f^j(p)},\]
we have that $\kappa(\mu_p)<\lambda(\mu_p)$ if and only if $A^n(p)$ is a quasi-compact operator in the sense of Section \ref{sec: quasi-compact}. Indeed, from the spectral radius formula we have that
\[\ra(A^n(p))=\lim_{k\to \infty} \|A^{kn}(p)\|^{1/k}=e^{n\lambda(\mu_p)}.\]
Similarly, by Nussbaum's formula for the essential spectral radius \cite{N70} we have that
\[\res(A^n(p))=\lim_{k\to \infty} \|A^{kn}(p)\|_{ic}^{1/k}=e^{n\kappa(\mu_p)}.\]
Therefore, $\kappa(\mu_p)<\lambda(\mu_p)$ if and only if $\res(A^n(p))<\ra(A^n(p))$ proving our claim.

\subsection{Main result in the case of constant periodic data} We now present our main result in the case when the cocycle has constant periodic data. Recall that
$A\colon M\to \cB(X)$ is said to be an  \emph{$\alpha$-H\"{o}lder continuous map} if there  exists a constant $C>0$ such that
\begin{displaymath}
\|A(x)-A(y)\| \leq C d(x,y)^{\alpha},
\end{displaymath}
for all $x,y\in M$. 
\begin{theorem}\label{thm: main quasi-compact}  
Let $f\colon M\to M$ be a homeomorphism exhibiting the closing property and $A\colon M\to \cB^i(X)$ be an $\alpha$-H\"older continuous map.
 Given $k<\dim X$ and a sequence of real numbers $\lambda_1\geq \lambda_2\geq \cdots \geq \lambda_k\geq \lambda_{k+1}$, suppose that
 \begin{equation}\label{eq: lambda k+1 bigger kappa mu}
 \lambda_{k+1}>\sup_{\mu\in \cM_{erg}(f)} \kappa(\mu).
 \end{equation}
If $A$ has $(k+1)$-constant periodic data of type $(\lambda_1,\ldots,\lambda_k,\lambda_{k+1})$ and $\lambda_k>\lambda_{k+1}$, then $A$ admits a dominated splitting of index $k$. 
\end{theorem}

\begin{remark}
Observe that $\sup_{\mu\in \cM_{erg}(f)} \kappa(\mu)=-\infty$ whenever $A\colon M\to \cB_0(X)$, as observed in Section \ref{sec: quasicompact cocycles}.
In particular, condition \eqref{eq: lambda k+1 bigger kappa mu} is automatically satisfied whenever dealing with compact cocycles.
\end{remark}

\begin{remark}\label{rem: bochi}
In \cite{Boch}, Bochi presented an example where the previous result fails if we assume that $A$ is only continuous even in the much simpler case when $A$ takes values in $SL(2,\mathbb{R})$, highlighting the importance of the assumption that $A$ is H\"older continuous. 
\end{remark}

 We will now extend the scope of Theorem \ref{thm: main quasi-compact} to include cocycles whose periodic data is not necessarily constant, but only ``close'' to constant. For this, we introduce the following notion.

\subsection{$\delta$-narrow periodic data}
As in Section \ref{sec: constant pd qc}, given $A\colon M\to \cB(X)$, $n\in \N$ and $p\in \Fix(f^n)$, let us assume that $A^n(p)$ is a quasi-compact operator and  
denote by $(\gamma_i(p))_{i=1}^{b_p}$ with 
\[
|\gamma_1(p)|\geq|\gamma_2(p)|\geq |\gamma_3(p)|\geq \ldots  \]
the exceptional eigenvalues of $A^n(p)$ repeated according to multiplicity where $b_p\in \N\cup \{+\infty\}$.

\begin{definition}
Let $A\colon M\to \cB(X)$ be such that $A^n(p)$ is a quasi-compact operator for every $p\in \Fix(f^n)$ and $n\in \N$.
Given $\delta\geq 0$, $k\leq \dim X $ and a sequence of real numbers $\lambda_1\geq \lambda_2\geq \cdots \geq \lambda_k$ (observe that $k$ may be infinite), we say that the cocycle $A$ has \emph{$\delta$-narrow periodic data around $(\lambda_1,\ldots,\lambda_k)$} if $b_p\geq k$ for every $p\in \Per(f)$ and
\[e^{n(\lambda_i-\delta)}\leq  |\gamma_i(p)|\leq e^{n(\lambda_i+\delta)} \]
for all $1\leq i\leq k, n\in \N$ and $p\in \Fix(f^n)$.
\end{definition}

 Note that $A$ having $k$-constant periodic data of type $(\lambda_1,\ldots,\lambda_k)$ is equivalent to $A$ having $0$-narrow periodic data around $(\lambda_1,\ldots,\lambda_k)$.

\begin{theorem}\label{thm: main narrow}
Let $f\colon M\to M$ be a homeomorphism exhibiting the closing property and $A\colon M\to \cB^i(X)$ be an $\alpha$-H\"older continuous map. 
Given $k<\dim X$ and a sequence of real numbers $\lambda_1\geq \lambda_2\geq \cdots \geq \lambda_k\geq \lambda_{k+1}$, suppose that condition \eqref{eq: lambda k+1 bigger kappa mu} holds and that $A$ has $\delta$-narrow periodic data around $(\lambda_1,\ldots,\lambda_k,\lambda_{k+1})$ for some $\delta>0$. Then, if $\lambda_k>\lambda_{k+1}$ and $\delta>0$ is small enough, $A$ admits a dominated splitting of index $k$. 
\end{theorem}

Obviously, Theorem \ref{thm: main quasi-compact} is a particular case of Theorem \ref{thm: main narrow}. Nevertheless, we have chosen to present both versions because the proof for Theorem \ref{thm: main quasi-compact} is much cleaner and already explores the main ideas needed for the general case. In particular, after proving Theorem \ref{thm: main quasi-compact}, we will only explain how to proceed in the general case without giving full details.

\section{Preliminaries}\label{sec: prelim} 
In this section, we gather some useful preliminary results that are going to be used in the proofs of our main results.

\subsection{Gelfand numbers and volume growth} Let $T\in \cB(X)$. For each $k\in \N$ such that $k\le d:=\dim X$, we define the \emph{$k$-Gelfand number} as
\[
c_k(T):=\inf \left\lbrace \|T_{\mid V}\|; \; V\in \cG^{k-1}(X) \right\rbrace,
\]
with $c_1(T)=\|T\|$ and the \emph{maximal $k$-volume growth} by
\[
V_k(T)=\sup \left\lbrace \text{det}(T\mid_V); V\in \cG_k(X) \right\rbrace,
\]
where 
\begin{equation*}
\text{det}(T\mid_V)=
\left\{
	\begin{array}{ll}
		\frac{m_{TV}(T(B_V))}{m_V(B_V)}  & \mbox{if } T\mid_V \text{ is injective, } \\
		0 & \mbox{otherwise},\\
	\end{array}
\right.
\end{equation*}
and $m_V$ denotes the \emph{Haar measure} on the subspace $V$ normalized so that the unit ball $B_V$ in $V$ has measure given by the volume of the Euclidean unit ball in $\R^k$.  
We recall that in the particular case when $(X,\|\cdot\|)$ is a Hilbert space, $c_k(T)$ coincide with the $k$-th singular value of $T$ while $V_k(T)$ coincide with the product of the top $k$ singular values of $T$.

We note that $V_k(T)$ and $\prod_{j=1}^k c_j(T)$ may be interpreted as the growth rates of $k$-dimensional volumes spanned by $\{Tv_i\}_{i=1}^k$, where $v_i\in X$ are unit vectors and, up to a  multiplicative constant, they coincide as shown in the next result.

\begin{lemma}[{Lemma 2.22 in \cite{B16}}]\label{lemma: relation volume growth}
Given $k\in \N$ such that $k\le \dim X$, there exists $C>0$ (depending only on $k$) such that for every $T\in \cB(X)$,
\[\frac{1}{C}V_k(T)\leq \prod_{j=1}^k c_j(T)\leq C V_k(T).\]
\end{lemma}

We shall also need the following auxiliary  result. 
\begin{lemma}\label{lemma: subadditive}
For every $k\in \N$, the map $T\to V_k(T)$ is continuous  on $\cB(X)$. Moreover, it is  also 
submultiplicative, i.e. 
\[
V_k(TS)\le V_k(T)V_k(S) \quad \text{for every $T, S\in \cB(X)$.}
\]
\end{lemma}
\begin{proof}
The continuity of the map $T\to V_k(T)$ is established in~\cite[Lemma 2.20]{B16}, while the submultiplicativity property follows from \cite[Proposition 2.13(3)]{B16}.
\end{proof}
Finally, we observe that the Gelfand numbers are also continuous maps in $\cB(X)$.

\begin{lemma}\label{lem: gelfand Lipschitz}
For every $k\in \N$ and $T,S \in \cB(X)$,
\[|c_k(T)-c_k(S)|\leq \|T-S\|.\]
\end{lemma}
\begin{proof}
Observe that $c_k(T+S)\leq c_k(T)+\|S\|$ for any $T,S\in \cB(X)$. In particular,
\[c_k(T)=c_k(S-(S-T))\leq c_k(S)+\|T-S\|\]
and thus $c_k(T)- c_k(S)\leq \|T-S\|$. Similarly, $c_k(S)- c_k(T)\leq \|T-S\|$. These two observations combined imply the desired conclusion.
\end{proof}

\subsection{Characterization of dominated splittings} We now recall a characterization of dominated splitting for operator cocycles given in \cite[Theorem 2]{BM} which will play a fundamental role in our proof.

\begin{theorem}\label{thm:BM charact}
Let $f\colon M\to M$ be a homeomorphism of a compact topological space $M$ and $A \colon M  \to \cB^i(X)$ be a cocycle over $f$. Then the following statements are equivalent:
\begin{itemize}
\item
There exist constants $C>0$ and $\tau \in (0,1)$ such that
\begin{equation}\label{eq:gap theo charact}
 \max\{c_{k+1}(A^n(x)),c_{k+1}(A^n(f(x)))\} < C\tau^n c_{k}(A^{n + 1}(x))
 \end{equation}
for all $x \in M$ and $n \geq 1$;
\item
$A$ admits a dominated splitting of index $k$.
\end{itemize}
\end{theorem}
We observe that a similar result in the case of cocycles taking values in $GL(d,\R)$ was previously obtained in \cite{BG}. Moreover, in this particular setting, condition \eqref{eq:gap theo charact} is equivalent to the, a priori, simpler one
\begin{equation}\label{eq: gap finite dim}
c_{k+1}(A^n(x)) < C\tau^n c_{k}(A^n (x)).
\end{equation}
On the other hand, in the context of Theorem \ref{thm:BM charact}, condition \eqref{eq:gap theo charact} can not be replaced by \eqref{eq: gap finite dim}, as observed in \cite[Example 3]{BM}.

\subsection{Mutliplicative Ergodic Theorem} In this section we present a version of the Multiplicative Ergodic Theorem adapted to our setting. In particular, we do not state it in its most general form. The statement as we present it is a combination of the main results in \cite{B16,FLQ13,GTQ}. 

\begin{theorem}\label{theo: MET}
Let $f\colon M\to M$ be a homeomorphism of a compact metric space $(M,d)$, $\mu\in\cM_{erg}(f)$ an ergodic measure and $A \colon M  \to \cB(X)$ a continuous cocycle over $f$.
Then, there exists an $f$-invariant set $\cR^\mu \subset M$ with $\mu(\cR^\mu )=1$ such that the following holds:
\begin{enumerate}
\item for each $x \in \cR^\mu $, the $q$-dimensional volume growth rates
\begin{equation}\label{eq: def of lq}
l_q (\mu) = \lim_{n \to \infty} \frac{1}{n} \log V_q(A^n(x))
\end{equation}
exist and are constant over $\cR^\mu $;
\item defining the sequence $(\zeta_q(\mu))_{q \geq 1}$ by $\zeta_1(\mu) = l_1(\mu)$ and $\zeta_q(\mu) = l_q(\mu) - l_{q-1}(\mu)$ for $q > 1$, we have that 
\[\zeta_1(\mu) \geq \zeta_2(\mu) \geq \cdots \quad \text{ and } \lim_{q\to +\infty} \zeta_q(\mu) = \kappa(\mu).\]
Let us denote by $(\xi_j(\mu))_{j=1}^{t}$ with $t=t(\mu)\in \N\cup \{+\infty\}$ the distinct values (in decreasing order) of the sequence $\{\zeta_q(\mu)\}_q$ for which $\zeta_q(\mu)>\kappa(\mu)$;

\item for each $x\in \Reg$ and $i\in \mathbb N \cap [1, t]$, there is a unique and measurable decomposition
\begin{equation*}\label{os}
 X=\bigoplus_{j=1}^i E_j(x) \oplus V_{i+1}(x),
\end{equation*}
where $E_j(x)$ are  finite-dimensional subspaces of $X$ and $A(x)E_j(x)=E_j(f(x))$. Furthermore, $V_{i+1}(x)$ are closed subspaces of $X$ and $A(x)V_{i+1}(x)\subset V_{i+1}(f(x))$;
\item for each  $x\in \Reg$  and   $v\in E_j(x)\setminus \{0\}$, we have 
\[
 \lim_{n\to \infty} \frac 1 n \log \lVert A^n(x)v\rVert =\xi_j(\mu).
\]
In addition, for every $v\in V_{i+1}(x)$,
\[
 \limsup_{n\to \infty} \frac 1 n \log \lVert A^n(x)v\rVert \le \xi_{i+1}(\mu).
\]
 \end{enumerate}
The numbers $\xi_i(\mu)$ are called  \emph{exceptional Lyapunov exponents} of the cocycle $A$ with respect to $\mu$ and the dimensions $d_i(\mu)=\dim E_i(x)$ are called \emph{multiplicities} of $\xi_i(\mu)$. In addition, the decomposition~\eqref{os} is called the \emph{Oseledets splitting}.
\end{theorem}
As observed in \cite[Remark 3.3]{B16}, there may be only finitely many distinct values of $\zeta_q(\mu)$ and, in this case, the last of these numbers is equal to $\kappa(\mu)$.

\section{Proof of Theorems \ref{thm: main quasi-compact} and \ref{thm: main narrow}}\label{sec: proof theo compact}
In this section we present the proofs of Theorems \ref{thm: main quasi-compact} and \ref{thm: main narrow} starting with the first one. 

\subsection{Proof of Theorem \ref{thm: main quasi-compact}} 
In order to prove Theorem \ref{thm: main quasi-compact}, all we have to do is to show that the hypothesis of the aforementioned theorem implies that condition \eqref{eq:gap theo charact} is satisfied. We split the proof in several auxiliary results. In what follows, we will use the letter $C$ to denote a positive constant that may differ in each step. Given $q\in \N$ and $\mu\in \cM_{erg}(f)$, recall the definition of $l_q(\mu)$ given in \eqref{eq: def of lq}.

\begin{lemma} \label{lem: infinite} 
Let $\mu\in \cM_{erg}(f)$ and $q\leq \dim X $.
If $\{\mu _j \}_{j\in \N}$ is sequence of ergodic $f$-invariant probability measures converging in the weak$^{\ast}$ topology to $\mu$, then
\begin{displaymath}
	\limsup _{j\to \infty} l_q(\mu _j) \leq l_q(\mu).
\end{displaymath}
\end{lemma}

\begin{proof}
For each $n\in\N$, let $\varphi _n \colon M\to [-\infty ,+\infty )$ be the map given by
\begin{displaymath}
	\varphi _n (x)=\dfrac{1}{n}\log V_q(A^n(x)).
\end{displaymath}
Then, since $\mu$ is ergodic, $\varphi _1^+ (x)=\max \{0,\varphi _1(x)\} \in \mbox{L}^1(\mu)$ and $\{n\varphi _n\}_{n\in \N}$ is a subadditive sequence (recall Lemma \ref{lemma: subadditive}), it follows by Kingman's Subadditive Theorem that
\begin{displaymath}
	l_q (\mu)=\lim _{n\to \infty} \varphi _n(x)=\inf _n \int \varphi _n (x)d\mu
\end{displaymath}
for $\mu$ almost every $x\in M$.
Similarly,
\begin{displaymath}
	l_q(\mu _j)=\inf _n \int \varphi _n (x)d\mu _j.
\end{displaymath}
Therefore, in order to complete the proof it is enough to show that
\begin{equation*}
\limsup _{j\to +\infty} \inf _n \int \varphi _n (x)d\mu _j\leq l_q(\mu).
\end{equation*}

For $m\in \N$, let us consider $\varphi _{n,m}\colon M\to (-\infty, +\infty)$ given by
\begin{displaymath}
	\varphi _{n,m}(x)=\left\{\begin{array}{cc}
		\varphi _n(x) & \mbox{if} \quad \varphi _n(x)\geq -m\\
		-m & \mbox{if}\quad \varphi _n(x)< -m.
	\end{array}
	\right.
\end{displaymath}
It follows easily from the definition and Lemma \ref{lemma: subadditive} that, for every $m,n\in \N$, $\varphi _{n,m}:M\to (-\infty, +\infty)$ is a continuous function. Moreover, 
\begin{displaymath}
\varphi _{n,m}(x)\geq \varphi _{n,m+1}(x) \quad \mbox{and} \quad \varphi _{n,m}(x)\xrightarrow{m\to +\infty} \varphi _{n}(x)
\end{displaymath}
for every $m,n\in \N$ and $x\in M$. 

By the Monotone Convergence Theorem we get
\begin{displaymath}
\int 	\varphi _n d\mu = \lim _{m\to \infty} \int \varphi _{n,m} d\mu =\inf _{m} \int \varphi _{n,m} d\mu.
\end{displaymath}
On the other hand, since $\mu _j \xrightarrow{w^{\ast}}\mu$ as $j$ goes to infinite and $\varphi _{n,m}$ is continuous, we have
\begin{displaymath}
	\int \varphi _{n,m}d\mu =\lim _{j\to \infty} \int \varphi _{n,m}d\mu _j \quad \forall m\in \N.
\end{displaymath}
Thus, combining these observations we get that
\begin{equation*}
\inf _n \int \varphi _n d\mu =\inf _n \left\{\inf _m \left[ \lim _{j\to \infty} \int \varphi _{n,m} d\mu _j \right]\right\}.
\end{equation*}
Now, for each $j$ and $m$ in $\N$ we have $\int \varphi _{n,m} d\mu _j \geq \inf _m\left\{ \int \varphi _{n,m} d\mu _j\right\}$ and thus
\begin{displaymath}
	\lim _{j\to \infty}\int \varphi _{n,m} d\mu _j \geq \limsup _{j\to \infty} \inf _m\left\{ \int \varphi _{n,m} d\mu _j\right\}
\end{displaymath}
for every $m\in \N$. Consequently,
\begin{equation}
\inf _m \left\{ \lim _{j\to \infty} \int \varphi _{n,m} d\mu _j \right\} \geq  \limsup _{j\to \infty} \left\{ \inf _m \int \varphi _{n,m} d\mu _j \right\}.
\label{eq: infinite LE estimative 1}
\end{equation}
Moreover, once again by the Monotone Convergence Theorem, 
\begin{equation*}
 \limsup _{j\to \infty} \left\{ \inf _m \int \varphi _{n,m} d\mu _j \right\}= \limsup _{j\to \infty}  \int \varphi _n d\mu _j.
\end{equation*}
Finally, proceeding as in \eqref{eq: infinite LE estimative 1} we get that 
\begin{equation*}
\inf _n \left\{ \limsup _{j\to \infty}  \int \varphi _n d\mu _j\right\} \geq \limsup _{j\to \infty} \inf _n \int \varphi _n d\mu _j
\end{equation*}
which combined with previous observations implies that 
\begin{displaymath}
	l_q(\mu)=\inf _n \int \varphi _n d\mu \geq \limsup _{j\to \infty} \inf _n \int \varphi _n d\mu _j
\end{displaymath}
completing the proof of the lemma.
\end{proof}

\begin{remark}
Note that the previous result holds even in the case when $l_q(\mu)=-\infty$.
\end{remark}

\begin{corollary}\label{cor: lk+1 finite}
For every $0\leq q\leq k+1$,
\[l_{q}(\mu)\geq \lambda_1+\ldots+\lambda_q \]
for every $\mu \in \cM_{erg}(f)$.
\end{corollary}
\begin{proof}
Suppose there exist $0\leq q\leq k+1$ and $\mu \in \cM_{erg}(f)$ such that $l_{q}(\mu)< \lambda_1+\ldots+\lambda_q$. 

Since $f$ has the closing property, we know that $\mu$ may be approximated by periodic measures in the weak$^\ast$ topology. That is, there exists a sequence of positive integers $(n_j)_{j\in \mathbb{N}}$ and periodic points $(p_j)_{j\in \N}$ with $f^{n_j}(p_j)=p_j$ for every $j\in \mathbb{N}$ such that the ergodic periodic measures given by 
\[\mu _{p_j} =\dfrac{1}{n_j}\sum _{i=0}^{n_j-1}\delta _{f^i(p_j)}\]
converge to $\mu$ in the weak$^*$ topology. Consequently, by Lemma \ref{lem: infinite}, it follows that 
\[\limsup_{j\to +\infty} l_{q}(\mu_j)<\lambda_1+\ldots+\lambda_q.\]
On the other hand, since $A$ has $(k+1)$-constant periodic data of type $(\lambda_1,\ldots,\lambda_k,\lambda_{k+1})$, it follows that $\lambda_1,\ldots,\lambda_k$ and $\lambda_{k+1}$ are the $(k+1)$-largest Lyapunov exponents of $(A,f,\mu_{p_j})$. Thus, by Theorem \ref{theo: MET}, we have that
\[\lambda_1+\lambda_2+\cdots+\lambda_{q}=\lim_{n\to +\infty}\frac{1}{n}\log V_{q}(A^n(p_j))=l_{q}(\mu_j)\]
for every $j\in \N$. Combining these observations we get a contradiction which proves the corollary.
\end{proof}

\begin{corollary}\label{cor: lyap exp limit vol growth}
The cocycle $A$ is quasi-compact with respect to every ergodic measure $\mu\in \cM(f)$. Moreover, the $(k+1)$-largest Lyapunov exponents of $A$ with respect to any such $\mu$ are $\lambda_1,\ldots,\lambda_k$ and $\lambda_{k+1}$ and
\[\lambda_1+\lambda_2+\cdots+\lambda_{q}=l_q(\mu)=\lim_{n\to +\infty}\frac{1}{n}\log V_{q}(A^n(x))\]
for $\mu$-almost every $x\in M$ and every $1\leq q\leq k+1$. 
\end{corollary}

\begin{proof} Let $\mu\in \cM_{erg}(f)$.
Recall that $c_1(A^n(x))=\|A^n(x)\|$. Then, using Lemma \ref{lemma: relation volume growth}, we get that
\[\begin{split}
\lambda(\mu)=\lim_{n\to +\infty}\frac{1}{n}\log c_{1}(A^n(x))
=\lim_{n\to +\infty}\frac{1}{n}\log V_{1}(A^n(x))
=l_1(\mu)
\end{split}\]
for $\mu$-almost every $x\in M$. 
In particular, by Corollary \ref{cor: lk+1 finite} and condition \eqref{eq: lambda k+1 bigger kappa mu}, 
\[\lambda(\mu)=l_1(\mu)\geq \lambda_1>\lambda_{k+1}>\kappa(\mu)\]
and $A$ is quasi-compact with respect to $\mu$. Moreover, recalling Theorem \ref{theo: MET}, we have that $\zeta_1(\mu)=l_1(\mu)=\lambda(\mu)>\kappa(\mu)$.
Hence, $\zeta_1(\mu)$ is an exceptional Lyapunov exponent of the cocycle $A$ with respect to $\mu$ and thus, by \cite[Theorem 2.5]{BD}, it may be approximated by the largest Lyapunov exponent of $A$ with respect to measures concentrated on periodic points. Therefore, since the largest Lyapunov exponent of $A$ with respect to any periodic measure is $\lambda_1$, we get that $\zeta_1(\mu)=\lambda_1$ and, consequently,
\[\lambda_1=l_1(\mu)=\lim_{n\to +\infty}\frac{1}{n}\log V_{1}(A^n(x))\]
for $\mu$-almost every $x\in M$.

We will now proceed recursively. Given $1<q\leq k+1$, suppose we have shown that, for every $1\leq j<q$, $\zeta_j(\mu)=\lambda_j$ and
\[\lambda_1+\lambda_2+\cdots+\lambda_{j}=l_j(\mu)=\lim_{n\to +\infty}\frac{1}{n}\log V_{j}(A^n(x))\]
for $\mu$-almost every $x\in M$. Thus, using Theorem \ref{theo: MET} and Corollary \ref{cor: lk+1 finite}, 
\[\lambda_1+\ldots+\lambda_{q-1}+\zeta_q(\mu)=\zeta_1(\mu)+\ldots+\zeta_{q-1}(\mu)+ \zeta_q(\mu) =l_q(\mu)\geq \lambda_1+\ldots+\lambda_{q}.\]
In particular, by \eqref{eq: lambda k+1 bigger kappa mu}, $\zeta_q(\mu)\geq \lambda_q\geq\lambda_{k+1}>\kappa(\mu)$ and $\zeta_q(\mu)$ is an exceptional Lyapunov exponent of the cocycle $A$ with respect to $\mu$. Thus, by \cite[Theorem 2.5]{BD}, $\zeta_q(\mu)$ may be approximated by the $q$-th largest Lyapunov exponent of $A$ with respect to measures concentrated on periodic points which we know is $\lambda_q$. Consequently, $\zeta_q(\mu)=\lambda_q$ and 
\[\lambda_1+\ldots+\lambda_q=l_q(\mu)=\lim_{n\to +\infty}\frac{1}{n}\log V_{q}(A^n(x))\]
for $\mu$-almost every $x\in M$. This completes the proof of the corollary. 
\end{proof}

\begin{remark}
We observe that, in the previous proof, it was important to ensure that the $(k+1)$-largest Lyapunov exponents of $A$ with respect to any ergodic measure are exceptional in order to be able to apply the results from \cite{BD} since, in general, for Banach cocycles, Lyapunov exponents may not be approximated by Lyapunov exponents on periodic orbits as observed in \cite{KS2}. It is only in this step of the proof that we \emph{explicitly} use condition \eqref{eq: lambda k+1 bigger kappa mu}.
\end{remark}

\begin{lemma}\label{lem: upper bound Vk}
For every $\varepsilon>0$ and $1\leq q\leq k+1$, there exists $C=C_{\varepsilon,q}>0$ such that
\begin{equation}
V_q(A^n(x))\leq C e^{(\lambda_1+\cdots+\lambda_q+\varepsilon)n}
\end{equation}
for every $x\in M$ and $n\in \N$.
\end{lemma}
For the proof we need the following result.
\begin{proposition}[Proposition 4.9 of \cite{KS13}]\label{prop: BS13}
Let $f\colon M\to M$ be a homeomorphism of a compact metric space $(M,d)$ and
$a_n \colon M\to \R$, $n\in \N$ be a subadditive sequence of continuous functions. Given $\mu\in \cM_{erg}(f)$, consider
\begin{equation*}\label{eq: aux BS}
\chi(\mu)=\lim_{n\to +\infty} \frac{a_n(x)}{n} \;\text{ for  $\mu$-a.e. $x\in M.$}
\end{equation*}
If $\chi(\mu)<0$ for every $\mu\in\cM_{erg}(f)$, then there exists $N\in \N$ such that $a_N (x) < 0$ for all $x\in M$.
\end{proposition}

\begin{proof}[Proof of Lemma \ref{lem: upper bound Vk}]
Given $n\in \N$, let us consider
\[a_n(x)=\log V_q(A^n(x))-(\lambda_1+\cdots+\lambda_q+\varepsilon)n.\]
Then, since $A$ is a continuous map, it follows by Lemma \ref{lemma: subadditive} that $(a_n(x))_{n\in \N}$ is a subadditive sequence of continuous functions. Moreover, by Corollary \ref{cor: lyap exp limit vol growth},
\[
\begin{split}
\lim_{n\to +\infty} \frac{a_n(x)}{n}&=\lim_{n\to +\infty} \frac{V_q(A^n(x))-(\lambda_1+\cdots+\lambda_q+\varepsilon)n}{n}\\
&=\lambda_1+\cdots+\lambda_q-(\lambda_1+\cdots+\lambda_q+\varepsilon)=-\varepsilon<0
\end{split}
 \]
for $\mu$-almost every $x\in M$ and every $\mu\in \cM_{erg}(f)$. Thus, by Proposition \ref{prop: BS13}, it follows that there exists $N\in \N$ such that $a_N(x)<0$ for every $x\in M$. In particular, 
\[V_q(A^N(x))\leq e^{(\lambda_1+\cdots+\lambda_q+\varepsilon)N}\]
for every $x\in M$. Thus, using again Lemma \ref{lemma: subadditive} and taking
\[C= \max_{1\leq j\leq N} \left\{ \max_{x\in M}\{V_q(A^{j}(x))e^{-(\lambda_1+\cdots+\lambda_q+\varepsilon)j}\},1\right\}, \]
it follows that
\[V_q(A^n(x))\leq Ce^{(\lambda_1+\cdots+\lambda_q+\varepsilon)n}\]
for every $x\in M$ and $n\in \N$ as claimed.
\end{proof}

\begin{lemma}\label{lem: lower bound Vk per}
For every $1\leq q\leq k+1$, there exists $C=C_{q}>0$ such that
\begin{equation}
V_q(A^n(p))\geq C^{-1} e^{(\lambda_1+\cdots+\lambda_q)n}
\end{equation}
for every $p\in \Fix(f^n)$ and $n\in \N$.
\end{lemma}
\begin{proof}
By Lemma \ref{lemma: relation volume growth}, it is enough to show that $c_j(A^n(p))\geq e^{\lambda_j n}$ for every $1\leq j\leq q$. Given $1\leq j\leq q$, let us consider
\[\mathcal{N}_j=\{v\in X: \; A^n(p)v=\gamma_i(p)v \text{ for some } 1\leq i\leq j  \}.\]
Then, $\dim(\mathcal{N}_j)\geq j$. In particular, any subspace $V$ of codimension $j-1$ must intersect $\cN_j$ and thus $\|A^n(p)_{\mid{V}}\|\geq |\gamma_j(p)|$. Consequently
\[c_j(A^n(p))\geq |\gamma_j(p)|=e^{\lambda_jn}\]
as claimed.
\end{proof}

\begin{lemma}\label{lem: lower bound Vk per gamma}
For every $\varepsilon,\gamma>0$ sufficiently small, $1\leq q\leq k+1$
and some fixed $l\in \N$, there exists $C=C_{\varepsilon,q,l}>0$ such that for every $n\in \N$, $p\in \Fix(f^n)$ and $\gamma n-1\leq i\leq \gamma n+l$ we have that
\begin{equation*}
V_q(A^i(p))\geq C^{-1} e^{(\lambda_1+\cdots+\lambda_q-\varepsilon)\gamma n}.
\end{equation*}
\end{lemma}
\begin{proof} 
Let $n\in \N$ be large enough so that $\gamma n +l\leq n$ and $\gamma n >1$ and take $i\in \N$ such that $\gamma n-1\leq i\leq \gamma n+l$. Then, by Lemmas  \ref{lemma: subadditive}, \ref{lem: upper bound Vk} and \ref{lem: lower bound Vk per} there exists $C>0$ such that
\[\begin{split}
C^{-1} e^{(\lambda_1+\cdots+\lambda_q)n}&\leq V_q(A^n(p))\leq V_q(A^{n-i}(f^i(p)))V_q(A^i(p))\\
&\leq C e^{(\lambda_1+\cdots+\lambda_q+\varepsilon)(n-i)}V_q(A^i(p))
\end{split}
\]
and
\[\begin{split}
V_q(A^i(p))&\geq C^{-2} e^{(\lambda_1+\cdots+\lambda_q+\varepsilon) i}e^{-\varepsilon n} .
\end{split}
\]
Consequently,
\[\begin{split}
V_q(A^i(p))&\geq C^{-2} e^{(\lambda_1+\cdots+\lambda_q+\varepsilon) (\gamma n-1)}e^{-\varepsilon n} 
\end{split}
\]
if $\lambda_1+\cdots+\lambda_q\geq 0$ or 
\[\begin{split}
V_q(A^i(p))&\geq C^{-2} e^{(\lambda_1+\cdots+\lambda_q+\varepsilon) (\gamma n+l)}e^{-\varepsilon n} 
\end{split}
\]
if $\lambda_1+\cdots+\lambda_q<0$. In any case, there exists $C>0$ (maybe depending also on $l$) such that
\[\begin{split}
V_q(A^i(p))&\geq C^{-2} e^{(\lambda_1+\cdots+\lambda_q) \gamma n}e^{-(1-\gamma)\varepsilon n} .
\end{split}
\]
Finally, changing $\varepsilon$ by $\gamma \varepsilon/(1-\gamma)$ we get that there exists $ C>0$ such that
\[V_q(A^i(p))\geq  C^{-1} e^{(\lambda_1+\cdots+\lambda_q-\varepsilon)\gamma n}\]
as claimed.  
\end{proof}

\begin{corollary}\label{cor: upper and lower bound ck per gamma}
For every $\varepsilon,\gamma>0$ sufficiently small, $1\leq q\leq k+1$
and some fixed $l\in \N$, there exists $C=C_{\varepsilon,q,l}>0$ such that for every $n\in \N$, $p\in \Fix(f^n)$ and $\gamma n-1\leq i\leq \gamma n+l$, we have that
\begin{equation*}
 C^{-1} e^{(\lambda_q-\varepsilon)\gamma n}\leq c_q(A^i(p))\leq C e^{(\lambda_q+\varepsilon)\gamma n}.
\end{equation*}
\end{corollary}
\begin{proof}
From Lemmas \ref{lem: upper bound Vk} and \ref{lem: lower bound Vk per gamma} it follows that there exists a constant $C>0$ such that
\[C^{-1} e^{(\lambda_1+\cdots+\lambda_q-\varepsilon)\gamma n}\leq V_q(A^i(p))\leq C e^{(\lambda_1+\cdots+\lambda_q+\varepsilon)\gamma n}\]
and 
\[C^{-1} e^{(\lambda_1+\cdots+\lambda_{q-1}-\varepsilon)\gamma n}\leq V_{q-1}(A^i(p))\leq C e^{(\lambda_1+\cdots+\lambda_{q-1}+\varepsilon)\gamma n}.\]
On the other hand, from Lemma \ref{lemma: relation volume growth} we have that there exists $C>0$ such that
\[C^{-1}c_q(A^i(p))\leq \frac{V_{q}(A^i(p))}{V_{q-1}(A^i(p))}\leq C c_q(A^i(p)). \]
Combining these observations we get the desired conclusion.
\end{proof}

\begin{lemma}\label{lem: upper lower bound ck general}
For every $\varepsilon>0$ sufficiently small and $1\leq q\leq k+1$, there exists $C=C_{\varepsilon}>0$ such that for every $n\in \N$ and $x\in M$, we have that
\begin{equation*}
 C^{-1} e^{(\lambda_q-\varepsilon)n}\leq c_q(A^n(x))\leq C e^{(\lambda_q+\varepsilon)n}.
\end{equation*}
\end{lemma}

\begin{proof}
Take $x\in M$ and $n\in \N$. By the closing property, there exists $c,\theta, l>0$ (independent of $x$ and $n$), $0\leq j \leq l$ and $p\in \Fix(f^{2n+j})$ such that
\[d(f^i(f^{-n}(x)),f^i(f^{-n}(p)))\leq c e^{-\theta \min\lbrace i, 2n-i\rbrace}d(f^{-n}(x),f^{n}(x)),
\]
for every $i=0,1,\ldots ,2n$. In particular,
\begin{equation}\label{eq: closing aux}
d(f^i(x),f^i(p))\leq C e^{-\theta (n-i)}
\end{equation}
for every $i=0,1,\ldots ,n$ and some $C>0$ independent of $x$ and $n$.

Our objective now is to estimate the distance between $A^m(p)$ and $A^m(x)$ for some special $m\in \N$. For this purpose, we follow very closely the idea developed in the proof of \cite[Proposition 3.4]{DG}.
Let us fix constants $ \varepsilon_0,\kappa,\eta>0$ and $\gamma \in (0,1/(l+1))$ such that 
\begin{equation}\label{eq: bound A}
\|A(x)\|\leq e^\eta \text{ for every $x\in M$}
\end{equation}
 and 
\[ \eta \gamma -\alpha \theta (1-\gamma)+\gamma\kappa <-\varepsilon_0. \]
Note that the last inequality is equivalent to
\begin{equation}\label{eq: rel gamma epsilon0}
0<\gamma <\frac{-\varepsilon_0+\alpha \theta}{\eta+\alpha\theta+\kappa}. 
\end{equation}
In particular, taking $\varepsilon_0>0$ small enough, we can take $\gamma\in (0,1/(l+1))$ such that the previous inequality holds.

Given $0\leq i\leq n$, let us consider $E_{i,n}:=A(f^i(x))-A(f^i(p))$. Since $A$ is $\alpha$-H\"older continuous, using \eqref{eq: closing aux} we get that
\begin{equation}\label{eq: est Ein}
\begin{split}
\|E_{i,n}\|& \leq Ce^{-\theta  \alpha ( n-i)}
\end{split}
\end{equation}
for some $C>0$ and every $i=0,1,\ldots,n$. Then, we can write
\[\begin{split}
A^{\lfloor \gamma n \rfloor}(x)&=\prod_{i=0}^{\lfloor \gamma n \rfloor-1}A(f^i(x))=\prod_{i=0}^{\lfloor \gamma n \rfloor-1}(A(f^i(p))+E_{i,n}).
\end{split}\]
Expanding the right hand side of this expression as a product and grouping together the terms involving exactly $t$ copies of the error terms $E_{i,n}$ with $0\leq i\leq \lfloor \gamma n \rfloor$, and denoting it by $B_{t,n}$, we get that
\[\begin{split}
A^{\lfloor \gamma n \rfloor}(x)&=A^{\lfloor \gamma n \rfloor}(p) +\sum_{t=1}^{\lfloor \gamma n \rfloor}B_{t,n}.
\end{split}\] 
Now, since each $B_{t,n}$ is formed by $\binom{\lfloor \gamma n \rfloor}{t}$ summands each of them involving exactly $t$ copies of the error terms $E_{i,n}$ with $0\leq i\leq \lfloor \gamma n \rfloor$  and $\lfloor \gamma n \rfloor-t$ terms of the form $A(f^i(p))$, it follows by \eqref{eq: bound A} and \eqref{eq: est Ein} that 
\[\|B_{t,n}\|\leq C \binom{\lfloor \gamma n \rfloor}{t}e^{-t\theta  \alpha (1-\gamma) n}e^{\eta (\gamma n-t)}\]
for some constant $C>0$. By \cite[Claim 3.5]{DG}, there exists $C_\kappa>0$ such that $ \binom{\lfloor \gamma n \rfloor}{t} \leq C_\kappa e^{\lfloor \gamma n \rfloor t \kappa}$. Thus, combining these observations we get that
\[\|B_{t,n}\|\leq C e^{-t\theta  \alpha (1-\gamma) n}e^{\eta (\gamma n-t)}e^{ \gamma n t \kappa}\]
for some constant $C>0$. Consequently,
\[
\begin{split}
\left\|\sum_{t=1}^{\lfloor \gamma n \rfloor}B_{t,n}\right\|&\leq \sum_{t=1}^{\lfloor \gamma n \rfloor} C e^{-t\theta  \alpha (1-\gamma) n}e^{\eta (\gamma n-t)}e^{ \gamma n t \kappa}\\
&=Ce^{\eta \gamma n}\sum_{t=1}^{\lfloor \gamma n \rfloor} e^{-t(\theta  \alpha (1-\gamma) n+\eta - \gamma n \kappa)}\\
&\leq Ce^{\eta \gamma n} \frac{e^{-(\theta  \alpha (1-\gamma) n+\eta - \gamma n \kappa)}}{1-e^{-(\theta  \alpha (1-\gamma) n+\eta - \gamma n \kappa)}}\\
&\leq \tilde C e^{\left(\eta\gamma -\theta  \alpha (1-\gamma) +\gamma \kappa\right)n}
\end{split}
\]
for some constant $C>0$.
Therefore, since $\eta \gamma -\alpha \theta (1-\gamma)+\gamma\kappa <-\varepsilon_0$, combining the previous observations we get that
\[\|A^{\lfloor \gamma n \rfloor}(x)-A^{\lfloor \gamma n \rfloor}(p)\|\leq Ce^{-\varepsilon_0 n}\]
where the constant $C$ does not depend on $x$ nor $n$.

Now, given $m\in \N$ take $n=\lceil \gamma^{-1}m\rceil$. In particular, $m=\lfloor \gamma n \rfloor$ and from the previous inequality it follows that
\[\|A^{m}(x)-A^{m}(p)\|\leq Ce^{-\varepsilon_0 \gamma^{-1} m}\]
for some $C>0$ and some $p\in \Fix(f^{2\lceil \gamma^{-1}m \rceil+j})$ with $0\leq j\leq l$ which combined with Lemma \ref{lem: gelfand Lipschitz} implies that
\begin{equation}\label{eq: aux ck}
|c_q(A^m(x))-c_q(A^m(p))|\leq Ce^{-\varepsilon_0 \gamma^{-1} m}.
\end{equation}
On the other hand, given $\varepsilon, \gamma>0$, by Corollary \ref{cor: upper and lower bound ck per gamma} (applied to $\gamma/2$),  there exists $C=C_{\varepsilon, \gamma}>0$ such that
\begin{equation*}
C^{-1} e^{(\lambda_q-\varepsilon) \frac{\gamma}{2} (2\lceil \gamma^{-1}m \rceil+j) }\leq c_q(A^i(p))\leq C e^{(\lambda_q+\varepsilon) \frac{\gamma}{2} (2\lceil \gamma^{-1}m \rceil+j) }
\end{equation*}
for $\frac{\gamma}{2} (2\lceil \gamma^{-1}m \rceil+j)-1 \leq i\leq \frac{\gamma}{2} (2\lceil \gamma^{-1}m \rceil+j) +l$.  
In particular, taking $i=m$ we get that there exists $C>0$ such that
\begin{equation}\label{eq: aux ck2}
C^{-1} e^{(\lambda_q-\varepsilon) m }\leq c_q(A^m(p))\leq C e^{(\lambda_q+\varepsilon) m }.
\end{equation}
Finally, since $\varepsilon_0$ may be taken as small as we want without changing $\gamma$ (recall \eqref{eq: rel gamma epsilon0}), combining \eqref{eq: aux ck} and \eqref{eq: aux ck2} we conclude the proof of the lemma.
\end{proof}

\begin{proof}[Proof of Theorem \ref{thm: main quasi-compact}]
We observe that, by Lemma \ref{lem: upper lower bound ck general}, there exists $C>0$ such that 
\[c_k(A^{n+1}(x))\geq C^{-1} e^{(\lambda_k-\varepsilon)(n+1)}\]
and 
\[c_{k+1}(A^n(x))\leq C e^{(\lambda_{k+1}+\varepsilon)n}.\]
Thus,
\[\frac{c_{k+1}(A^n(x))}{c_k(A^{n+1}(x))}\leq \frac{Ce^{(\lambda_{k+1}+\varepsilon)n}}{C^{-1} e^{(\lambda_k-\varepsilon)(n+1)}}\leq C^2e^{-(\lambda_k-\varepsilon)} e^{-(\lambda_k-\lambda_{k+1}-2\varepsilon)n}.\]
Now, since $\lambda_k>\lambda_{k+1}$, for $\varepsilon>0$ small enough we have that  $\tau:=e^{-(\lambda_k-\lambda_{k+1}-2\varepsilon)}\in (0,1)$. Consequently, taking the new constant $C$ as $C^2e^{-(\lambda_k-\varepsilon)}$, there exists $\tau\in(0,1)$ such that 
\[c_{k+1}(A^n(x))\leq C \tau^n c_k(A^{n+1}(x)).\]
Similarly, we can show that there exists $\tau\in (0,1)$ and $C>0$ such that
\[c_{k+1}(A^n(f(x)))\leq C \tau^n c_k(A^{n+1}(x)).\]
Therefore, combining these observations with Theorem \ref{thm:BM charact} we conclude the proof of Theorem \ref{thm: main quasi-compact}.
\end{proof}

\begin{remark}\label{rem: injetcitvity hyp}
We observe that the hypothesis that each $A(x)$ is injective is used only to apply the results of Blumenthal and Morris \cite{BM} in the last step of the proof. In particular, all the auxiliary results up to Lemma \ref{lem: upper lower bound ck general} are still valid without this assumption.
\end{remark}

\subsection{Proof of Theorem \ref{thm: main narrow}}
The proof of Theorem \ref{thm: main narrow} is similar to the proof of Theorem \ref{thm: main quasi-compact} except that now we have an extra parameter to consider, namely, $\delta$. As a consequence, in all the estimates where previously appeared some $\lambda_j$, now will also appear a term involving $\delta$. For instance, in Corollary \ref{cor: lk+1 finite}, we will now get an estimate of the form
\[l_{q}(\mu)\geq \lambda_1+\ldots+\lambda_q -q\delta\]
for every $\mu \in \cM_{erg}(f)$ and $0\leq q\leq k+1$. In Corollary \ref{cor: lyap exp limit vol growth}, taking $\delta>0$ small enough so that $\lambda_{k+1}-2(k+1)\delta >\sup_{\mu\in \cM_{erg}(f)}\kappa(\mu)$ (this condition is used to ensure that the $(k+1)$-largest Lyapunov exponents of $A$ with respect to $\mu$ are exceptional and thus can be approximated by Lyapunov exponents on periodic points), we will get, for every $\mu \in \cM_{erg}(f)$, that
$\zeta_q(\mu)\in [\lambda_q-\delta,\lambda_q+\delta]$ and
\[l_q(\mu)=\lim_{n\to +\infty}\frac{1}{n}\log V_{q}(A^n(x))\in [\lambda_1+\cdots+\lambda_{q}-q\delta , \lambda_1+\cdots+\lambda_{q}+q\delta ]\]
for $\mu$-almost every $x\in M$ and every $1\leq q\leq k+1$. Similarly, in the estimates that follow, instead of terms like $\lambda_q$ and $\lambda_1+\ldots+\lambda_q$, we will have terms of the form $\lambda_q+N\delta$ and $\lambda_1+\ldots+\lambda_q +N\delta$ where $N\in \mathbb Z$ is uniformly bounded with bound independent of $p$, $x$ and $n$. In particular, taking $\delta$ sufficiently small, we can proceed with the estimates as in the proof of Theorem \ref{thm: main quasi-compact} and obtain that
\[c_{k+1}(A^n(x))\leq C \tau^n c_k(A^{n+1}(x)).\]
and
\[c_{k+1}(A^n(f(x)))\leq C \tau^n c_k(A^{n+1}(x)).\]
for some $\tau\in (0,1)$ and $C>0$ and every $x\in M$ and $n\in \N$. Thus, using Theorem \ref{thm:BM charact} we conclude the proof of Theorem \ref{thm: main narrow}.

\section{Applications}\label{sec: applications}
In this section we present some applications of Theorem \ref{thm: main quasi-compact} as well as consequences of its proof. 
In order to simplify exposition, we will assume that $A$ is a compact cocycle so that condition \eqref{eq: lambda k+1 bigger kappa mu} is automatically satisfied. Analogous versions can be also obtained in the general setting. Moreover, some of the results can be obtained as well in the setting of Theorem \ref{thm: main narrow}, but we refrain from writing them explicitly.

\begin{corollary}\label{cor: hyper}
Let $f\colon M\to M$ be a homeomorphism exhibiting the closing property and $A\colon M\to \cB_0^i(X)$ be an $\alpha$-H\"older continuous map. Given $1<q\leq \dim X$ and a sequence of real numbers $\lambda_1\geq \lambda_2\geq \cdots \geq \lambda_q$, suppose $A$ has $q$-constant periodic data of type $(\lambda_1,\ldots,\lambda_q)$. If there exists $1\leq k<q$ such that $\lambda_k>0>\lambda_{k+1}$, then $A$ is uniformly hyperbolic. More precisely, for each $x\in M$ there exist a splitting $X=\cE(x)\oplus \cF(x)$ and constants $C>0$ and $\tau\in (0,1)$ such that
\begin{equation}\label{eq: hyper contr}
\|A^n(x)v\|\leq C\tau^n \|v\| \;\text{ for every } v\in  \cF(x) \text{ and } n\in \N
\end{equation}
and
\begin{equation}\label{eq: hyp exp}
\|A^{-n}(x)u\|\leq C\tau^n \|u\| \;\text{ for every } u\in  \cE(x) \text{ and } n\in \N
\end{equation}
where
 \[A^{-n}(x)=(A^n(f^{-n}(x))\rvert \cE(x))^{-1}\colon \cE(x) \to \cE(f^{-n}(x)).\]
 \end{corollary}

\begin{proof}
We start observing that, by Theorem \ref{thm: main quasi-compact}, $A$ admits a dominated splitting of index $k$ which we will denote by $X=\cE(x)\oplus \cF(x)$ as in Section \ref{sec: dominated splitting}. Moreover, by Corollary \ref{cor: lyap exp limit vol growth}, we know that the $(k+1)$-largest Lyapunov exponents of $A$ with respect to any ergodic measure $\mu\in \cM(f)$ are $\lambda_1,\ldots,\lambda_{k+1}$. Thus, since the splitting $X=\cE(x)\oplus \cF(x)$ is dominated and has index $k$, we get that $\lambda_1,\ldots,\lambda_{k}$ are the Lyapunov exponents of $A(x)_{|\cE(x)}$ with respect to any ergodic measure $\mu$ while all the Lyapunov exponents of $A(x)_{|\cF(x)}$ with respect to any ergodic measure $\mu$ are smaller than or equal to $\lambda_{k+1}$. Furthermore, since $\dim \cE(x)=k$ and $A(x)\cE(x)=\cE(f(x))$ for every $x\in M$, we get that $A(x)_{|\cE(x)}$ is invertible for every $x\in M$ and it makes sense to consider 
 \[A^{-n}(x)=(A^n(f^{-n}(x))\rvert \cE(x))^{-1}\colon \cE(x) \to \cE(f^{-n}(x))\]
for $n\in\N$. In particular, from the previous observations together with the Multiplicative Ergodic Theorem for finite-dimensional cocycles (see \cite[Theorem 4.2]{LLE}), it follows that
\[\lim_{n\to +\infty}\frac{1}{n}\log \|A^{-n}(x)\|\leq -\lambda_k\]
for almost every $x\in M$ with respect to every ergodic measure $\mu\in \cM(f)$.

Let us now consider the function $a_n(x)=\log \|A^{-n}(x)\|+(\lambda_k-\varepsilon)n$. Then, $(a_n(x))_{n\in\N}$ is a subaditive sequence over $f^{-1}$ and, moreover, since the splitting $X=\cE(x)\oplus \cF(x)$ is continuous, each $a_n(x)$ is also a continuous function. Furthermore, by the observations above,
\[
\begin{split}
\lim_{n\to +\infty} \frac{a_n(x)}{n}&=\lim_{n\to +\infty} \frac{\log \|A^{-n}(x)\|+(\lambda_k-\varepsilon)n}{n}\\
&\leq -\lambda_k+(\lambda_k-\varepsilon)=-\varepsilon<0
\end{split}
 \]
for almost every $x\in M$ with respect to every ergodic measure $\mu\in \cM(f)$. Proceeding as in the proof of Lemma \ref{lem: upper bound Vk} we get that there exists $C>0$ such that
\[\|A^{-n}(x)\|\leq Ce^{(-\lambda_k+\varepsilon)n}\]
for every $n\in \N$. Thus, taking $\tau=e^{-\lambda_{k}+\varepsilon}$ with $\varepsilon>0$ small enough so that $\varepsilon<\lambda_{k}$ we get that \eqref{eq: hyp exp} is satisfied.

Similarly, considering $b_n(x)=\log \|A^n(x)_{\mid{\cF(x)}}\|-(\lambda_{k+1}+\varepsilon)n$ and proceeding as in the proof of Lemma \ref{lem: upper bound Vk}, we get that there exists $C>0$ such that
\[\|A^{n}(x)\|\leq Ce^{(\lambda_{k+1}+\varepsilon)n}\]
for every $n\in \N$. Thus, taking $\tau=e^{\lambda_{k+1}+\varepsilon}$ with $\varepsilon>0$ small enough so that $\varepsilon<-\lambda_{k+1}$, we get that \eqref{eq: hyper contr} is satisfied. The proof of the corollary is complete.
\end{proof}

\begin{corollary}
Let $f\colon M\to M$ be a homeomorphism exhibiting the closing property and $A\colon M\to \cB_0^i(X)$ be an $\alpha$-H\"older continuous map. Given $1<q\leq \dim X$ and a sequence of real numbers $\lambda_1\geq \lambda_2\geq \cdots \geq \lambda_q$, suppose $A$ has $q$-constant periodic data of type $(\lambda_1,\ldots,\lambda_q)$. If there exists $1\leq t<i<j\leq q$ such that $\lambda_t>0=\lambda_i>\lambda_{j}$, then $A$ is partially hyperbolic. More precisely, for each $x\in M$ there exist a \emph{dominated splitting} $X=\cE(x)\oplus \cH(x) \oplus \cF(x)$ and constants $C>0$ and $\tau\in (0,1)$ such that
\begin{equation*}
\|A^n(x)v\|\leq C\tau^n \|v\| \;\text{ for every } v\in  \cF(x) \text{ and } n\in \N
\end{equation*}
and
\begin{equation*}
\|A^{-n}(x)u\|\leq C\tau^n \|u\| \;\text{ for every } u\in  \cE(x) \text{ and } n\in \N
\end{equation*}
where $A^{-n}(x)$ is as in Corollary \ref{cor: hyper}.
Moreover,
\begin{equation*}
\lim_{n\to +\infty}\frac{1}{n}\log \|A^n(x)w\|=0\;\text{ for every } v\in  \cH(x)\setminus \{0\}.
\end{equation*}
 \end{corollary}

\begin{proof}
The dominated splitting with three bundles can be obtained by applying Theorem \ref{thm: main quasi-compact} twice: first,  taking $\lambda_k$ with $k=\max\{t: \lambda_t>0\}$ we obtain a splitting $X=\cE(x)\oplus \tilde \cF(x)$; next, taking $\lambda_k$ with $k=\max\{i: \lambda_i=0\}$ we obtain a splitting $X=\tilde \cE(x)\oplus \cF(x)$. Then, considering $\cH(x):=\tilde \cE(x)\cap \tilde \cF(x)$ we get the desired dominated splitting into three bundles. Exponential estimates along $\cE(x)$ and $\cF(x)$ can be obtained by proceeding as in the proof of Corollary \ref{cor: hyper}.
\end{proof}

\begin{corollary}\label{cor: conv ck}
Let $f\colon M\to M$ be a homeomorphism exhibiting the closing property and $A\colon M\to \cB_0(X)$ be an $\alpha$-H\"older continuous map. Given $1<q\leq \dim X$ and a sequence of real numbers $\lambda_1\geq \lambda_2\geq \cdots \geq \lambda_q$, suppose $A$ has $q$-constant periodic data of type $(\lambda_1,\ldots,\lambda_q)$. Then, for every $1\leq k\leq q$ we have that
\begin{equation}\label{eq: limits exist everywhere}
\lambda_k=\lim_{n\to +\infty} \frac{1}{n}\log c_k(A^n(x)) 
\end{equation}
\emph{for every} $x\in M$. Moreover, these limits are uniform on $x\in M$.
\end{corollary}
Before we go into the proof of the corollary, let us recall the following consequence of \cite[Theorem A.3]{Mor}.

\begin{proposition}\label{prop: unif conv}
If $a_n\colon M\to \R$ is a sequence of continuous and subadditive functions, then
\[\lim_{n\to +\infty} \sup_{x\in M}\frac{a_n(x)}{n}=\sup_{\mu \in \cM(f)} \inf_{n\in \N}\frac{1}{n}\int_M a_n(x)d\mu.\]
\end{proposition}

 We note that variations of Proposition \ref{prop: unif conv} were obtained previously by several different authors. We refer to \cite{Mor} and references therein for more on this subject.

\begin{proof}[Proof of Corollary \ref{cor: conv ck}]
Recall that, by Remark \ref{rem: injetcitvity hyp}, all the results up to Lemma \ref{lem: upper lower bound ck general} can be used in this setting. Thus, the fact that \eqref{eq: limits exist everywhere} holds for  
\emph{every} $x\in M$ follows directly from Lemma \ref{lem: upper lower bound ck general}. Now, this result combined with Lemma \ref{lemma: relation volume growth} implies that
\[\lim_{n\to +\infty} \frac{1}{n}\log V_k(A^n(x)) =\lambda_1+\ldots+\lambda_k
\]
for every $x\in M$. Then, recalling Lemma \ref{lemma: subadditive} and using Proposition \ref{prop: unif conv}, we get that
\[\lim_{n\to +\infty} \sup_{x\in M} \frac{1}{n}\log V_k(A^n(x))=\lambda_1+\ldots+\lambda_k \]
and, consequently, $\left(\frac{1}{n}\log V_n(x)\right)_{n\in \N}$ converges uniformly to $\lambda_1+\ldots+\lambda_k$. Using again Lemma \ref{lemma: relation volume growth} we get that the limit \eqref{eq: limits exist everywhere} is also uniform.
\end{proof}

In order to conclude the paper, let us make two simple observations. For this purpose, we will restrict ourselves to the case when $\dim X<+\infty$. Moreover, suppose we are in the hypothesis of Corollary \ref{cor: conv ck} and assume $A$ has $q$-constant periodic data with $q=\dim X$.
\begin{remark}
It was observed in \cite[Example 2.2]{Bac} that having all the Lyapunov exponents with respect to every measure $\mu\in \cM(f)$ uniformly bounded by bellow is not enough to guarantee that a semi-invertible cocycle is invertible. On the other hand, if the cocycle has $q$-constant periodic data with $q=\dim X$, we get from Lemma \ref{lem: upper lower bound ck general} (recall Remark \ref{rem: injetcitvity hyp}) that $A(x)$ is indeed invertible for every $x\in M$.
\end{remark}

\begin{remark}
Roughly speaking, a point $x\in M$ is said to be \emph{Lyapunov regular} if, at $x$, Lyapunov exponents are well-defined and behave predictably according to the Multiplicative Ergodic Theorem. In contrast, $x$ is said to be \emph{Lyapunov irregular} if it is not regular. As a consequence of the Multiplicative Ergodic Theorem, we know that $\mu$-almost every $x\in M$ is Lyapunov regular with respect to any measure $\mu\in\cM(f)$. On the other hand, in some contexts, the set of Lyapunov irregular points can be big in the topological sense as observed in \cite{Fur}. As a consequence of Corollary \ref{cor: conv ck} and, more generally, the proof of Theorem \ref{thm: main quasi-compact}, assuming $A$ has $q$-constant periodic data with $q=\dim X$, we get that the set of Lyapunov irregular points is empty.
\end{remark}

%%%%%%%%%%%%%%%%%%%%%%%%%%%%%%%%%%%%%%%%%%%%%%%%%%%%%%%%%
%%%%%%%%%%%%%%%%%%%%%%%%%%%%%%%%%%%%%%%%%%%%%%%%%%%%%%%%%

%\medskip{\bf Acknowledgments.}
%L.~Backes was partially supported by a CNPq-Brazil PQ fellowship under Grant No. 307633/2021-7.

%
%\vspace{0.1in}
%\medskip{\bf Statements and Declarations}
%\vspace{0.1in}
%
%\textbf{Competing Interests:} no potential conflict of interest was reported by the author.

%%%%%%%%%%%%%%%%%%%%%%%%%%%%%%%%%%%%%%%%%%%%%%%%%%%%%%%%%
%%%%%%%%%%%%%%%%%%%%%%%%%%%%%%%%%%%%%%%%%%%%%%%%%%%%%%%%%

\end{document}